\documentclass[12pt,a4paper,leqno]{amsart}
\usepackage[all]{xy}

\advance\textwidth by 3.5cm   
\advance\textheight by 3.5cm  
\advance\oddsidemargin by -1.6cm     
\advance\evensidemargin by -1.6cm

\usepackage{amsmath,amsthm,amssymb}      
\theoremstyle{plain}
\begingroup
\newtheorem{thm}{Theorem}%[section]

\newtheorem{prp}[thm]{Proposition}
\newtheorem{cor}[thm]{Corollary}
\endgroup
\theoremstyle{definition}
\newtheorem{dfn}[thm]{Definition}  
\newtheorem{xmp}[thm]{Example}
\newtheorem{rem}[thm]{Remark}

\newcommand\C{{\mathbb C}}

\newcommand\E{{\mathbb E}}

\newcommand\N{{\mathbb N}}
\newcommand\R{{\mathbb R}}

\newcommand\cA{{\mathcal A}}
\newcommand\cB{{\mathcal B}}
\newcommand\cC{{\mathcal C}}
\newcommand\cD{{\mathcal D}}
\newcommand\cE{{\mathcal E}}
\newcommand\cF{{\mathcal F}}
\newcommand\cH{{\mathcal H}}

\newcommand\cK{{\mathcal K}}
\newcommand\cL{{\mathcal L}}
\newcommand\cM{{\mathcal M}}

\newcommand\cP{{\mathcal P}}

\newcommand\cR{{\mathcal R}}
\newcommand\cS{{\mathcal S}}

\newcommand\cW{{\mathcal W}}
\newcommand\cX{{\mathcal X}}
\newcommand\cY{{\mathcal Y}}

\newcommand\ot{\otimes}
\newcommand\hot{\widehat{\otimes}}

  \newcommand\g{\gamma}
  \newcommand\G{\Gamma}
  \newcommand\s{\sigma}
  
        \renewcommand\a{\alpha} 
        \renewcommand\b{\beta} 
        \renewcommand\d{\delta}
        \renewcommand\l{\lambda}
        \renewcommand\L{\Lambda}
        \newcommand\e{\epsilon}
        \newcommand\Om{\Omega}
        \newcommand\om{\omega}

\newcommand\xchi{\smash{\raise 0.5ex\hbox{$\chi$}}}

\newcommand\f[1]{{\mathfrak #1}}

\numberwithin{equation}{section}
\title{The Cotlar-Stein Lemma, Grothendieck's Inequality and All That}

\author{Brian Jefferies}
  
\address{School of Mathematics\\ The University of New South Wales\\ NSW 2052
AUSTRALIA}
\email{b.jefferies@unsw.edu.au}

\date{\today}


\begin{document}

\begin{abstract} The purpose of this paper is point out connections between scattering theory,
double operator integrals, Krein's spectral shift function, integration theory, bimeasures,
Feynman path integrals, harmonic and functional analysis and many other applications to quantum physics
made since the last 50 ears or so. 

The starting point is Kluvanek's \textit{Integration Structures} which he
hoped to apply to quantum physics and is now bearing fruit from the contributions of many
authors, especially former Soviet mathematical physicists in the intervening years. Soon, a practical quantum field theory in four space-time dimensions
satisfying the Wightman axioms may be proved to exist. This is the aim of one of the Clay Prizes.
At the moment, only toy models exist in fewer that four space-time dimensions.
\end{abstract}

\maketitle
\baselineskip 16pt

\tableofcontents

%\cListoffigures                    % list of figures, optional

%\cListoftables                     % list of tables, optional

\setcounter{page}{1}
%\section{Introduction}DOI, bimeasures old 2000s, scattering theory 60s

\section{Integration Structures}
Let $(\Om,\cS,\mu)$ be a measure space, or more generally, let $\mu:\cS\to\C$ be a $\s$-additive
set function defined on a $\d$-ring of subset $\cS$ of $\Om$. In either case, we have the bound
$$\left|\int_\Om f\, d\mu\right| \le \|f\|_{L^1(|\mu|)}$$
for all $\mu$-integrable functions $f$. In case $\mu$ is scalar valued, $|\mu|$ is the variation of $\mu$.
On the other hand, if $\mu(E) \in \{0,\infty\}$ for all $E\in \cS$, then it is well known that $L^1(\mu) = \{0\}$. We say that the norm $\|\cdot\|_{L^1(\mu)}$ is 
\textit{integrating for $\mu$}.

I. Kluvanek generalised this concept in the following direction \cite{Klu}. Let $\cK$ be some family scalar valued functions
defined on a set nonempty set $\Om$ with the zero function $0\in\cK$. A \textit{gauge $\rho$ on $\cK$} is a function $\rho:\cK\to [0,\infty)$ such that $\rho(0) = 0$.
Let $\cR$ be a family of gauges on $\cK$. Then $\cR$ is said to be \textit{integrating for a map $\mu:\cK\to\C$} if the following holds:
\begin{itemize}
\item[a)] $f\in\cK$ and there exists $c_j\in\C$, $f_j\in\cK$, $j=1,2,\dots$,
$f(\om)=\sum_{j=1}^\infty c_jf_j(\om)$ for all $\om\in\Om$ such that $\sum_{j=1}^\infty |c_j|.|f_j(\om)| < \infty$
\item[b)] $\sum_{j=1}^\infty |c_j|.\rho(f_j) < \infty$ for all $\rho\in\cR$,
\item[c)] for every $\rho\in \cK$, there exists $C_\rho > 0$ such that $|\mu(f)| \le C_\rho\sum_{j=1}^\infty |c_j|.\rho(f_j)$.
\end{itemize}

\begin{dfn} If $\cR$ satisfies just a) and b) on $\cK$, then the quasicomplete lcs $L^1(\cK,\cR)$ consists of all $\cR$-equivalence classes
of functions $f:\Om\to\C$ such that a) and b) hold. The collection sim$(\cK)$ of all finite linear combinations of functions from $\cK$
is densely embedded in $L^1(\cK,\cR)$ \cite{Klu}. The theory mimics Lebesgue, Daniell and Stone integration \cite{Stone} quite closely but is richer, as the following examples show.
\end{dfn}

\begin{xmp} i) Let $\mu$ be Lebesgue measure on $\R$ and $\cK=\{\xchi_I: I = [a,b), a,b\in\R,\ a\le b\}$.
The length of an interval $I$ is denoted by $|I|$. Then $\rho = |\cdot|$ defines the gauge $|\cdot|:\cK\to [0,\infty)$ by
$|\xchi_I|=|I|$ that is integrating for $\mu$ on $\cK$. For any Lebesgue measurable function $f:\Om\to[0,\infty]$, we have
$$\int_\Om f\,d\mu = \inf\left\{\sum_{j=1}^\infty |c_j|.|I_j|: f(\om)=\sum_{j=1}^\infty c_j\xchi_{I_j}(\om)\ \forall
\om\in\Om\text{ with }\sum_{j=1}^\infty |c_j|.|f_j(\om)| < \infty\right\},$$

%This observation was apparently first made in 1948 in \cite{}. 

If the $c_j$, $j=1,2,\dots$ are elements of a Banach space $\cX$,
then we obtain the \textit{Bochner integral} \cite{Mik}.

ii) Let $E$ be a quasicomplete lcs (bounded sets are complete in the associated uniformity), $(\Om,\cS)$ a measurable space and $m:\cS\to E$ a countably additive measure on the $\s$-algebra $\cS$.
Then $m$ has a family $\Delta(m)$ of finite nonnegative measures that is integrating for every scalar measure $\langle m,\xi\rangle$, $\xi\in E'$
(Bartle-Dunford-Schwartz Theorem). 

Furthermore, for every continuous seminorm $p$ on $E$, there exists 
$C_p >0$ and $\mu_p\in \Delta$, such that $p(m)(S) \le C_p\mu_p(S)$ for all $S\in\cS$. Here
$$p(m)(S) =\sup_{\xi\in U_p^\circ}|\langle m,\xi\rangle|(S), \quad S\in\cS,$$
for $U_p =\{x\in E:p(x)< 1\}$ is the $p$-semivariation of $m$ and for any $m$-integrable function $f:\Om\to\C$ the inequalities
$$p\left(\int_S f\,dm\right)\le  C_p\mu_p(|f|)$$
obtain so that the collection $\Delta(m)$ is integrating for $m$. 

The lcs $L^1(m)$
of $m$ integrable functions is endowed with the topology induced by the continuous seminorms $f\mapsto p(m)(f)$
as $p$ ranges over the continuous seminorms of $E$. Unsurprisingly, if $E$ is a Fr\'echet space or Banach space, then so is $L^1(m)$.
The space $L^1(m)$ is quasicomplete exactly when the $\Delta(m)$-measure algebra of $\cS$ is \textit{Dedekind complete}.
It has only recently been verified by R. Becker, R. Ricker and S. Okada,  see \cite{RO}, 
that this occurs 
precisely when $m$ is totally absolutely
continuous with respect to a single localisable measure $\lambda$, that is, one for which $L^1(\l)$ is itself Dedekind complete.
The construction is due to pioneering work of I. Kluvanek %\cite{} 
on conical measures.
There was a gap in the proof for $\R^I$ filled by R. Becker. %\cite{}.

iii) For a spectral measure $P$ acting on a Banach space $\cX$, the gauge $\rho:f\mapsto \|P(|f|)\|_{\cL(\cX)}$, is an integrating
norm for $P$ so that $L^1(\rho) = L^1(P)$ as vector spaces \cite{Klu}. However, the usual topology of $L^1(P)$ as in ii) above is not a norm topology if $\cX$ is infinite dimensional, because $P$ is usually only $\s$-additive in the strong operator topology, except in matrix theory.
\end{xmp}

\begin{xmp} (One Dimensional Brownian Motion) Let $\Om = C([0,1])$, 
the continuous functions on the interval $[0,1]$. Let $X_t:\Om\to\R$ be the evaluation functions $X_t(\om) = \om(t)$
for $0\le t\le 1$ and $\om\in\Om$. Let $m([a,b))= X_b-X_a$ for all $0\le a\le b\le t$.

We seek a Borel probability measure $P$ on $\Omega$ with expectation operator $\E$  for which
$$\E|m([a,b))|^2 =b-a.$$
Then $m:\cB([0,1])\to L^2(P)$ is an \textit{orthogonally scattered vector measure} in the sense that
$$\lambda(A\cap B)= \int_\Om m(A)m(B)\,dP = 0\text{ if }A\cap B=\emptyset,\ A,B\in \cB([0,1]).$$
Moreover $\|m(A)\|^2 \le \lambda(A)$, the Lebesgue measure of $A\in \cS$ and 
$$\left\|\int_A f\,dm\right\|\le \lambda(|f|.\xchi_A)^\frac12,\quad A\in \cS$$
for $f\in L^1(m)$, so that $\lambda^{\frac12}$ is \textit{integrating for $m$} on \textit{subintervals} of $[0,1]$. 

The random variable
$\int_0^1 f\,dm \in L^2(P)$ is called the \textit{It\^o integral} of $f\in L^2([0,1])$ and the total integation map is the \textit{It\^o isometry}. The random variable $\int_0^1 f\,dm$ is usually written as $\int_0^1 f\,dX$ 
despite the observation that, with probability one, the continuous process $X$ has unbounded variation on every 
nonempty interval so that $dX$ cannot be a pointwise $\s$-additive with values in the random variables
$L^0(P)$.

A. Kolmogorov demonstrated the existence of a Baire probability measure $\tilde P$ on $\R^{[0.1]}$ with these properties
and N. Wiener proved that a probability measure $P$ with $P(E\cap \Om) = \tilde P(E)$ exists for cylinder sets $E$ contained in $\R^{[0.1]}$.
It follows that the quadratic variation process
$$[X]_t= \lim_{\cP_t}\sum_{[a,b)\in \cP_t} |X_a-X_b|^2$$
converges in probability $P$ over all partitions $\cP_t$ of $[0,t)$ and $\E[X]_t = t$.
The vector measure $m$  does however have a density $\Phi:[0,1]\to\cS'$ with respect to $\lambda$, with values in distributions $\cS'$.
The process $\Phi$ is usually called \textit{White Noise}.
\end{xmp}  

\section{Bimeasures}

Let $\cX,\ \cY$ be lcs and $(\L,\cE)$, $(\G,\cF)$ measurable spaces.
We write the semi-algebra of product sets $E\times F$, $E\in\cE$ and $F\in\cF$ as $\cE\times\cF$ 
and the algebra it generates as $\cE\ot \cF$.
The corresponding \textit{$\s$-algebra} is $\cE\hot \cF$. The space of continuous linear operators $u:\cX\to \cY$
is denoted by $\cL(\cX,\cY)$. We always give the vector space $\cL(\cX,\cY)$ or $\cL(\cX) := \cL(\cX,\cX)$
the topology of \textit{strong convergence} generated
by the seminorms $u\mapsto p(ux)$ for any $x\in\cX$ and continuous seminorm $p:\cY\to \R_+$.
Then $\cX$ is identical to the lcs $\cL(\C,\cX)$.

A separately
$\s$-additive set function $m:\cE\times\cF\to \cL(\cX)$
is called a \textit{bimeasure}, that is,
$$m(E\times F) =\sum_{j=1}^\infty m(G_j)$$
for $G_j\in \cE\ot \cF$, $j=1,2,\dots$, pairwise disjoint and $E\times F= \bigcup_{j=1}^\infty G_j$. 
This is the same as requiring the two set functions
$$E\longmapsto m(E\times F)\text{ and }F\longmapsto m(E\times F),\quad  E\in\cE,\ F\in\cF$$
to be $\s$-additive \cite{Klu}. It is well known that a bimeasure on $\cE\times\cF$ may not be 
the restriction of a measure on 
the $\s$-algebra $\cE\hot\cF$.

\begin{xmp}\label{xmp:1}
 Let $\varphi\in\ell^2$ and $m(E\times F) = \sum_{k\in F}\left( \sum_{j\in E} \varphi(j)e^{i\pi j}\right)\varphi(k)e^{-i\pi k}$, the convergence being in $\ell^2$ as Fourier series,
 \begin{equation}\label{eqn:1}
|m(E\times F)|\le \|\varphi\|_2^2,\quad E,F\subseteq \N.
\end{equation}
and by linearity, $|m(f\ot g)|\le \|\varphi\|_2^2\|f\|_\infty\|g\|_\infty$
for simple functions $f,g$ on $\N$.

But any additive extension of $m$ is unbounded on the $\s$-algebra $\cE\hot \cF$.
\end{xmp}

Extending to functions, we see that
$$\int_{\L\times\G} f\ot g\,dm = m(f\ot g),\quad f \in L^\infty(\L),\ g \in L^\infty(\G),$$
in the limit. There exist $C>0$ such that $|m(f\ot g)|\le C\|f\|_\infty\|g\|_\infty$
in the scalar case or for any continuous seminorm $p$ in the vector case., there exists $C_p > 0$ such that 
$$p\big(m(f\ot g)\big)\le C_p\|f\|_\infty\|g\|_\infty.$$
Bimeasures were examined in a series of papers by M. Morse and Transue.

A function $\varphi $ belonging to the projective tensor product $L^\infty(\L)\hot_\pi L^\infty(\G)$
has a uniform  expansion $\varphi =\sum_{j=1}^\infty f_j\ot g_j$
with $\sum_{j=1}^\infty \|f_j\|_{\infty} \|g_j\|_{\infty} < \infty,$
so 
$$\int_{\L\times\G} \varphi\,dm = \sum_{j=1}^\infty m(f_j\ot g_j).$$
Not much more can be said without additional assumptions.

On the other hand, in Example \ref{xmp:1}, the representation

\begin{equation}\label{eqn:Peller_rep}
\psi(j,k) = \sum_{t=1}^\infty \a(j,t)\b(k,t),\quad j,k=1,2,,\dots,
\end{equation}
enables us to write
$$\int_{\N\times\N} \psi\,dm = 
\sum_{t=1}^\infty m\big(\a(\cdot,t)\ot\b(\cdot,t)\big).$$
It can happen that $\int_{\N\times\N} |\psi|\,d|m| =\infty$
because $|m|(E\times F) =\sum_{j\in E,k\in F}|\varphi(j)||\varphi(k)|$.

Applying inequality (\ref{eqn:1}), we have
$$\left|\sum_{t=1}^\infty m\big(\a(\cdot,t)\ot\b(\cdot,t)\big)\right|
\le\|\varphi\|_2^2\sum_{t\in T} \|\a(\cdot,t)\|_{\ell^\infty}\|\b(\cdot,t)\|_{\ell^\infty},$$
so $\sum_{t\in T} \|\a(\cdot,t)\|_{\ell^\infty}\|\b(\cdot,t)\|_{\ell^\infty}<\infty$
ensures that $\psi$ is integrable.

For $\psi$ represented by formula (\ref{eqn:Peller_rep}), let us define
\begin{align}
\|[\psi]\|_{L^1_G(m)}
= \left\|\left(\sum_{t\in\N}  |\a(\cdot,t)|^2\right)^\frac12\right\|_{L^\infty(\N)}
\left\|\left(\sum_{t\in\N} |\b(\cdot,t)|^2\right)^\frac12\right\|_{L^\infty(\N)}.\label{eqn:Peller_norm}
\end{align}

We see that the inequality
$$\left|\int_{\N\times\N} \psi\,dm\right| \le \|[\psi]\|_{L^1_G(m)}$$
characterises $m$-integrability, even though $m$ is not a genuine $\s$-additive measure. This is essentially \textit{Grothendieck's inequality} established  in 1958 (see \cite{Pis}).

\begin{dfn} The smallest positive number $K_G$  for which
$$\sum_{t\in T} \|\a(\cdot,t)\|_{\ell^\infty}\|\b(\cdot,t)\|_{\ell^\infty} 
\le K_G\left\|\left(\sum_{t\in\N}  |\a(\cdot,t)|^2\right)^\frac12\right\|_{\ell^\infty}
\left\|\left(\sum_{t\in\N} |\b(\cdot,t)|^2\right)^\frac12\right\|_{L^\infty(\N)}$$
is called \textit{Grothendieck's constant}.
\end{dfn}
Consequently, $\|[\psi]\|_{L^1_G(m)} < \infty$ ensures that 
$\psi$ is $m$-integrable and the Banach function space norm $[\psi]\mapsto\|[\psi]\|_{L^1(m)}$
is integrating for the simple scalar bimeasure $m$, as is required for a decent 
theory of integration in the sense of Lebesgue and modern Harmonic Analysis.

It turns out that the essential property of the simple  scalar bimeasure $m$ above
is that it is a well defined measure on compact sets, that is to say, on finite sets.

\begin{dfn} Let $\L$, $\G$ be Hausdorff topological spaces
with Borel $\s$-algebras $\cE=\cB(\L)$, $\cF=\cB(\G)$.

A separately
$\s$-additive set function $m:\cE\times\cF\to \cL(\cX)$
is called a \textit{regular (Radon) bimeasure} if
\begin{itemize}
\item[a)] $m(\cdot,F)$ and $m(E,\cdot)$ are compact inner regular
for every $E\in\cE$ and $F\in\cF$, and
\item[b)] $E\times F\mapsto m\big((E\cap K_1)\times(F\cap K_2)\big)$
$E\in\cE$, $F\in\cF$ is the restriction to $\cE\times\cF$ of a finite Radon measure on $\L\times\G$ for each compact set $K_1,\ K_2$.
\end{itemize}
By virtue of Bourbakist principles, the integral $E\times F\mapsto \int_{E\times F}\varphi\, dm$ is defined to be the unique regular bimeasure $\varphi.m$ equal to $\varphi.m$ on compact sets, that is, on
$$(\cE\cap K_1)\ot  (\cF\cap K_2)$$
for all compact sets $K_1,K_2$. The product $\varphi.m$ of $\varphi$ with $m$, the indefinite integral 
$$\int\varphi\,dm:E\times F\mapsto \int_{E\times F} \varphi\,dm := (\varphi.m)(E\times F),\quad E\in\cE,\ F\in\cF,$$
is itself a regular bimeasure as one would hope.
\end{dfn}
With no topology, we can just use a compact class $\cK_j$ of sets $K_j$, $j=1,2$. \cite{Klu}.
Even if $\cX$ is a Banach space, the Banach function space $\{[\psi]:\|[\psi]\|_{L^1_G(m)} < \infty\}$
need not be the optimal space of $m$-integrable functions because we deal with the strong operator
topology of $\cL(\cX)$ rather then the \textit{uniform} operator topology
of uniform convergence on the closed unit ball of $\cX$. In this setup,
the Cotlar-Stein Lemma to be mentioned later characterises $L^1(m)$ and its integrating norm topology
when $m=QP$ for spectral measures $Q$ qnd $P$.

Let $\cM(\cF)$ denote the scalar valued measures on the $\s$-algebra $\cF$ with the total variation norm.
Then for a scalar bimeasure $m$ the set function $E\mapsto m(E,\cdot)$, $E\in\cE$, is a measure-valued measure
that is $\s$-additive for the topology of setwise convergence, equivalent to the weak topology $\s(\cM(\cF),\cL^\infty(
G))$ on $\cM(\cF)$ by uniform convergence. It follows that the \textit{semivariation norm}
$$\|m\| =\sup\{|m(f\ot g)|: \|f\|_\infty \le 1,\ \|g\|_\infty \le 1\}$$
is a norm on the space of scalar bimeasures as well as for $\cL(X)$-valued bimeasures.
The corresponding space is $L^1(\|m\|(\,\cdot\,))= L^\infty(Q)\hot_\pi L^\infty(P)$.

For a regular bimeasure $m:\cE\times\cF\to \cL(\cX)$, we shall determine $L^1(m)$ and its locally convex topology $\tau_1(m)$.
The semivariation norm $\|\cdot\|$ determines a norm topology on $L^1(m)$ that is strictly stronger than
$\tau_1(m)$ if $\cK$ is infinite dimensional. The situation is a standard feature of vector measure theory and
generalises to polymeasures.

The natural choice for $\tau_1(m)$ is to take fundamental compact classes $\cK_j$, $j=1,2$ and take 
the topology generated by $L^1$-topology for the strong operator valued measures 
$$\{(\xchi_{K_1}\ot \xchi_{K_2}).m:K_j\in\cK_j,\ j=1,2\}$$
using the inner regularity of $m$ together with the semivariation norms $\|mx\|$ for $x\in\cX$. 
The topology first described in \cite{JefRick} was motivated by quantum mechanics and the bimeasure $QP$ described below.
\section{Spectral Measures}

We now assume that $Q,P$ are regular spectral measures and $m=QP$, that is,
$$m(E\times F) = Q(E)P(F),\quad E\in\cE,\ F\in\cF.$$

A special role is played by the position operator $Q(E):\psi\mapsto \xchi_E.\psi$, $E\in\cB(\R^d)$,
$\psi\in L^2(\R^d)$ in quantum mechanics in dimension $d=1,2,\dots$\,. For the unitary Fourier transform
$$\f{F}: L^2(\R^d)\to L^2(\R^d)$$
given by 
$$\f{F}f(\xi) =\hat f(\xi) := {(2\pi)^{-d/2}}\int_{\R^d}e^{-i\langle\xi, x\rangle}f(x)\, dx,\quad \xi\in\R^d,$$
for $f \in L^1(\R^d)$, we let $P(E) = \f{F}^*Q(E)\f{F}$, $E\in\cB(\R^d)$, be the corresponding momentum operator (``Questions")
and $P := \f{F}^*Q\f{F}:E\mapsto \f{F}^*Q(E)\f{F}$, $E\in\cB(\R^d)$, the spectral measure of the position operator
in quantum mechanics over $L^2(\R^d)$. Then the definite integral
$$\int_{\R^d\times\R^d}\s\,d(QP): L^2(\R^d)\to L^2(\R^d)$$
is the pseudodifferential operator with symbol $\s:\R^d\times \R^d\to\C$ given by
$$\left(\int_{\R^d\times\R^d}\s\,d(QP)\psi\right)(x) = {(2\pi)^{-d/2}}\int_{\R^d}e^{i\langle x,\xi\rangle}\s(x,\xi)\hat \psi(\xi)\,d\xi,\quad\psi\in L^2(\R^d),$$
when $\s$ is rapidly decreasing. The Cotlar-Stein result characterises the class $L^1(QP)$ of symbols $\s$ determining a bounded
pseudodifferential operator $\int_{\R^d\times\R^d}\s\,d(QP)$ on $L^2(\R^d)$, although it is usually not formulated in this manner.
It is easy to check the basic operator equality
$$Q(E)\left(\int_{\R^d\times\R^d}\s\,d(QP)\right)P(F) = \int_{E\times F}\s\,d(QP),\quad E,\, F\in\cB(\R^d).$$
A similar treatment applies to the space $L^1(PQ)$. Below we shall see that
the classical Weyl quantisation procedure is not a bimeasure $W$, but it is finitely additive.
This is actually a difficult problem in harmonic analysis and singular integral operators.
The question is resolved by checking whether or not $\|W(g,E)\| = \|W(E,g)\|$
is bounded uniformly for all finite unions $E$ of intervals and all uniformly bounded
sets of measurable functions $g$.

\section{$L^1(QP)$ and the Cotlar-Stein Lemma}

The argument in the last section suggests that to compute the indefinite integral $\varphi.(QP)$, 
if it exists at all, it suffices to show that the operator $(QP)(\varphi)\in \cL(\cH)$ exists, a standard feature of spectral theory
in its many guises. The Cotlar-Stein Lemma allows us to compute $(QP)(\varphi)$ and hence, $L^1(QP)$ with its topology.

The Cotlar-Stein Lemma concerning matrices tells us that
\begin{align*}
\sum_j\|Q(|f_k|^2)P(|g_j|^2\|^\frac12 \le M&,\  \sum_k\|Q(|f_k|^2)P(|g_j|^2)\|^\frac12 \le M \implies \\
&\left\|\sum_k (QP)(f_k\ot g_k)\right\| \le M.
\end{align*}
The best reference is Terry Tao's Blog\footnote{Lemma 1,
https://terrytao.wordpress.com/2011/05/25/the-cotlar-stein-lemma} and
the Wikipedia article\footnote{\hbox{https://en.wikipedia.org/wiki/Cotlar\%E2\%80\%93Stein\_lemma]}}.

Now define the Banach function space norm

\begin{align*}
\|\varphi\|_{L^1(QP)}
 &= \\
&\sup\left\{\sum_j\|Q(|f_k|^2)P(|g_j|^2\|^\frac12,\ \sum_k\|Q(|f_k|^2)P(|g_j|^2)\|^\frac12:
\sum_{j,\,k}|f_k|\ot|g_j|\le |\varphi|  \right\}.
\end{align*}
It follows that $L^1(QP)$ is a Banach function space in which $L^1_G(QP)$ is strictly embedded, because
$$\max\left\{\sum_j(\sup|f_k|^2))^\frac12(\sup|g_j|^2)^\frac12,\ \sum_k(\sup|f_k|^2))^\frac12(\sup|g_j|^2)^\frac12: j,k \right\}$$
$$= \max\left\{\|f_k\|_\infty\sum_j \|g_j\|_\infty   , \|g_j\|_\infty\sum_k \|f_k\|_\infty  : j,k\right\}$$
for $\sum_{j,\,k}|f_k|\ot|g_j|\le |\varphi|$
and this obviously differs from the $L^1_G(QP)$-norm which is stronger.

 Write this set as $\f C(|\varphi|)$. Then
$$\left\|(QP)(|\varphi|)\right\| = \sup_{u\in \f C(|\varphi|)} \left\|(QP)(u)\right\|.$$
Standard Banach lattice arguments ensure that $L^1(QP)$ is a Dedekind complete Banach function space
and $(QP):L^1(QP)\to\cL(\cH)$ is a contimuous linear map for the uniform operator topology.
In fact, for regular spectral measures $Q,P$, the lcs topology $\tau_1(QP)$ mentioned above possesses the remarkable property that
the topology defined 
by the norm $\varphi\mapsto \left\|(QP)(|\varphi|)\right\|$,
which is an integrating norm for $QP$ on the linear space of  simple functions on product sets, has the same bounded subsets of $L^1(QP)$ determined by $\tau_1(QP)$. The identity
$$\varphi.(QP) = \varphi_+\cdot(QP) - \varphi_-\cdot(QP)$$
analogous to the Lebesgue theory does not rely on $Q$ and $P$ being regular and
$$\left\|(QP)(|\varphi|)\right\|\le \inf\left\{\sum_{t\in T} \|\a(\cdot,t)\|_{L^\infty(Q)}\|\b(\cdot,t)\|_{L^\infty(P)} : 
\varphi(\l,\g)= \sum_{t\in T} \a(\l,t)\b(\g,t)\right\} $$
because, clearly
$$\left\|\sum_{t\in T} Q(\a(\l,t))P(\b(\g,t))\right\|_{\cL(\cH)}\le \sum_{t\in T} \|\a(\cdot,t)\|_{L^\infty(Q)}\|\b(\cdot,t)\|_{L^\infty(P)}$$
and the inequality may be strict. The integrals $\varphi.(QP)$ determined by $L^1(QP)$ and the locally convex topology
$\tau_1(QP)$ are the same and are identical as vector spaces---the norm topology of $L^1(QP)$ and the locally convex topology
$\tau_1(QP)$ have the same bounded sets by the uniform boundedness principle.

Although $QP$ is not a vector valued measure, the Cotlar-Stein lemma ensures that it possesses
properties remarkably similar to a Banach space valued measure. The same analysis applies to the polymeaures
considered in Remark \ref{rem:5} below arising in quantum theory.

\section{The Weyl Functional Calculus}

Let $\cM,\cD$ be a $(2n)$-system of $S_\om(\C)$-sectorial operators and $\cW_{\cM,\cD}$ the associated bilnear
$H^\infty$-Weyl functional calculus, when square function estimates obtain
\cite{Jef0}. 

We  extend the linear map $f\mapsto \cW_{\cM,\cD}(f)$
to a class larger than holomorphic functions by the Cotlar-Stein Lemma, so that
$$\cD{om}(\cM,\cD) = \left[(H^\infty\otimes L^\infty)\bigcup( L^\infty\otimes H^\infty)\right]$$
is optimal or equivalently, the linear space
 $$\left[({\bf sim}(S_\om(\C^n))\otimes L^\infty(\C^n))\bigcup( L^\infty((\C^n))\otimes {\bf sim}(S_\om((\C^n))))\right]$$
 is optimal by extendibility.
The convergence of the  associated bilinear singular integral was a longstanding conjecture in Harmonic Analysis
solved eventually be Michael Lacey, but optimality is actually
a simple application of Cotlar's matrix bounds and McIntosh methods. Moreover, holomorphic functions with decay in $S_\om(\C^n)$ suffice if there are no square function estimates.

The map $\cW_{\cM,\cD}$ is not a bimeasure or separately $\s$-additive,\,
but $f\mapsto \cW_{\cM,\cD}(f\circ diag)$ is an $H^\infty(S_\om)$-functional calculus corresponding to
the $n$-system $\cW_{\hbox{ diag($\cM,\cD$)}}$ projected onto the diagonal, so
giving the finitely additive ``spectral measure" on $S_\om(\C)$ of the sectorial operator $\frac12(\langle\cM,\cD\rangle
+\langle\cD,\cM\rangle)$. For the classical system, this directly defines a genuine selfadloint
spectral measure by the classical vector measure extension theory.

Although $(E,F)\mapsto \cW_{\cM,\cD}(\xchi_{E\times F})$ is finitely additive on
intervals as a bilinear singular integral operator, it is not \textit{regular} on product sets
and therefore not separately $\s$-additive, so it
is not a genuine bimeasure and the domain
$\cD{om}(\cM,\cD)$ is genuinely optimal for the continuous linear map
$\cW_{\cM,\cD}:\langle{\bf sim}(\R^{2n}),\|\cdot\|_\infty\rangle\to \cL(\R^{2n})$
in the selfadjoint case.
The beast reared its ugly head in 1994 \cite{JefRick} and is laid to rest in 2023 here.

However, for the classical Weyl system and generalisations, results on Fourier integral operators
provide an integration structure for $\cW_{\cM,\cD}$
and an optimal domain via deep Harmonic Analysis methods, as would be expected \cite{vanNPort}.

\section{Double Operator Integrals}
Let $\cH$ be a separable Hilbert space. Two bounded linear operators $S,T\in \cL(\cH)$ determine
a \textit{transformer} $({S\overline\ot T})_{\cC_p(\cH)}:u\mapsto SuT$ for every element of the Schatten class
$\cC_p(\cH)$, $1\le p\le \infty$, an old idea going back to M. Krein and Yu Daletskii
in the 50's  \cite{{BirmSol}}. By duality 
$$\|({S\overline\ot T})_{\cC_p(\cH)}\|_{\cL(\cC_p(\cH))} = \|({S\overline\ot T})_{\cC_q(\cH)}\|_{\cL(\cC_q(\cH))}$$
for the dual index $q$ satisfying $1/p+1/q=1$ with $1\le p,q\le\infty$. In particular, for the trace class operators 
$\cC_1(\cH)$ on $\cH$, $\langle \cC_1(\cH),\cC_\infty(\cH)\rangle$ is a dual pair for the compact linear operators
$\cC_\infty(\cH) = \cC_1(\cH)'$ \cite{BirmSol}. The crude inequality $\|ST\|_{\cL(\cH)} \le
 \|({S\overline\ot T})_{\cC_1(\cH)}\|_{\cL(\cC_1(\cH))}$ helps us estimate a proper subspace of $L^1(QP)$
 because $QP = QIP$ for the identity operator $I$ on $\cH$, see M. Birman and M. Solomyak \cite[Section 3.1]{BirmSol}
 and V.Peller. %\cite{}.

The following result was proved by M. Birman and M. Solomyak \cite[Section 3.1]{BirmSol}
who initiated this program in scattering theory %\cite{}.

\begin{thm} Let $(\L ,\cE)$ and $(\G,\cF)$ be measurable spaces and $\cH$ a separable Hilbert space.  
Let $P:\cE\to\cL _s(\cH)$ and $Q:\cF\to\cL _s(\cH)$ be  spectral measures. 

Then  
there exists a unique spectral measure $({P\overline\ot Q})_{\cC_2(\cH)}:\cE\ot\cF \to \cL (\cC_2(\cH))$ such that
$$({P\overline\ot Q})_{\cC_2(\cH)}(A)= ({P\ot Q})_{\cC_2(\cH)}(A)$$ 
for every set $A$ belonging to the algebra
$\cA $ of all finite unions of product sets $E\times F$ for $E\in\cE$, $F\in\cF$, and
$$\int_A\varphi\,d(P\ot Q)_{\cC_2(\cH)} = \int_A\varphi\,d({P\overline\ot Q})_{\cC_2(\cH)}\in \cL (\cC_2(\cH)),\quad A\in \cE\ot\cF,$$  
for every bounded $(\cE\ot\cF)$-measurable function $\varphi:\cL \times M\to \C$. Moreover,
$$\|({P\overline\ot Q})_{\cC_2(\cH)}(\varphi)\|_{\cL (\cC_2(\cH))} = \|\varphi\|_\infty.$$
\end{thm}
For spectral measures $P$ and $Q$, the formula
$$\left(\int_{E\times F}\varphi\,d(P\ot Q)_{\f{S}}\right)T = \left(\int_{\L \times \G}\varphi\,d(P\ot Q)_{\f{S}}\right)P(E)TQ(F)$$
holds for each $E\in\cE$, $F\in\cF$ and $T\in\f{S}$, so it is only necessary to verify that 
$$\int_{\L \times \G}\varphi\,d(P\ot Q)_{\f{S}}\in\cL (\f{S})$$
in order to show that $\varphi$ is $(P\ot Q)_{\f{S}}$-integrable.

The following observation is just a multiliear extension of the concepts above.
Now for simplicity, we take the \textit{Fourier transform} of $f\in L^1(\R)$ to be the function $\hat f:\R\to \C$
defined by
$\hat f(\xi)=\int_\R e^{-i\xi x }f(x)\,dx$ for $\xi\in \R$.

\begin{rem}\label{rem:5} Although we shall not consider \textit{multiple operator integrals}, as mentioned above,
transformers like $(Q\ot \cdots \ot Q)_{\cC_2(\cH)}$ are relevant to \textit{Feynman path integrals} in theoretical physics \cite{Jef}. Suppose that
$S$ is a C$_0$-semigroup of Hilbert-Schmidt operators or operators
in the Schatten class $\cC_p(\cH)$ for some $1 < p <\infty$. Then
for every $0<t_1<\cdots < t_n < t$ the expression
$$
(Q\ot \cdots \ot Q)_{\cC_2(\cH)}
(S(t_n-t_{n-1}),\dots,S(t_1))= 
QS(t_n-t_{n-1})\cdots QS(t_1)Q
$$
defines a separately $\s$-additive set function
$$F_0\times F_1\times\cdots\times F_n\longmapsto Q(F_n)S(t_n-t_{n-1})\cdots Q(F_1)S(t_1)Q(F_0),\quad F_j\in\cF,\ j=1,\dots,n.$$
that is the resriction of an operator valued measure.
It is possible to make sense of the inclusions
$$ 
 L^\infty(M\times \cdots \times M) = L^1\big((Q\ot \cdots \ot Q)_{\cC_2(\cH))}\big)  \subset L^1\big(QS(t_n-t_{n-1}) \cdots QS(t_1)Q\big) $$
and
$$L^1\big((Q\ot \cdots \ot Q)_{\cC_p(\cH))}\big)  \subset  L^1\big(QS(t_n-t_{n-1}) \cdots QS(t_1)Q\big)
.$$
In the case $p\ne 2$, the inclusions provide a criterion for integration with respect to operator valued set functions that may only be separately $\s$-additive and not fully $\s$-additive on the algebra of product sets (polymeasures) giving elementary Feynman-type integrals.
Multiple operator integrals are considered in greater detail in  \cite{SkTo}.

The group $S(t)=e^{-it\Delta/2}$, $t\in\R$, of bounded linear operators
on $L^2(\R^d)$ belongs to no $p$-Schatten class for $p\neq\infty$ and $L^1\big(QS(t_n-t_{n-1}) \cdots QS(t_1)Q\big)
$ is best described in terms of symbol classes of certain oscillatory pseudo-differential operators as mention above for the bimeasure $QP$. %\cite{}. 
If the operator valued Feynman set function $M_{-i}^t$ described in \cite{Jef} is restricted to cylinder sets 
$$\cE_{t_1,\dots,t_n} =\{X_{t_1}\in B_1,\dots,X_{t_n}\in B_n\}$$
 with  $0 < t_1 <\cdots < t_n \le t$ %\cite{}
 , then
 $$L^1(M_{-i}^t\restriction\cE_{t_1,\dots,t_n}) = L^1\big(QS(t_n-t_{n-1}) \cdots QS(t_1)Q\big)$$
 is completely understood by the Cotlar-Stein Lemma as a Banach function space similar to the Lebesgue space $L^1(\mu)$
 for an abstract measure $\mu$. This is a small step towards treating Feynman's mathematical and emotional inadequacies,
 which seems to have been missed until now.

On the other hand, product functions
$\varPhi=f_0\ot f_1\ot\cdots \ot f_n$, $f_j\in L^\infty((M,\cF,Q))$, $j=1,\dots,n$,
always belong to 
$$L^1\big(QS(t_n-t_{n-1}) \cdots QS(t_1)Q\big)$$
so that
\begin{align*}
\int_{M\times\cdots\times M} \varPhi\,d\big(QS(t_n-t_{n-1}) \cdots QS(t_1)Q\big)
&= Q(f_n)S(t_n-t_{n-1}) \cdots Q(f_1)S(t_1)Q(f_0)\\
&=\int_\Omega \varPhi\circ X\,dM_{-i}^t
\end{align*}
as expected. The space of paths $\Om$ consists of continuous functions $\om$ on $[0,\infty)$
with $X_t(\om) = \om(t)$ for $t\ge 0$ as for Brownian motion.

Subspaces of the symbol classes $L^1((QP)\stackrel{n}{\cdots} (QP))$ and $L^1((PQ)\stackrel{n}{\cdots}(PQ))$ are also described in terms of the Cotlar-Stein lemma in an analogous procedure to that above, in the case that $Q$ and $P$ are regular spectral measures. We use the integrating seminorms
$$\rho_h^{(n)}(f) = \|\left(QP\cdots QP\right)(|f|)h\|,\quad h\in \cH,\quad \rho_h = \rho_h^{(1)}$$
on product functions $f$. Note that the spectral measure property implies
the useful characterisation
$$f\in L^1(QP) \iff     \xchi_Ef\xchi_F\in L^1(QP),\quad \forall E,F,$$
$Q(E)(QP)(f)Q(F) = (QP)(\xchi_Ef\xchi_F)$
and
$$ L^\infty(Q)\hot_\pi L^\infty(P) \subset L^1((QP)_{\f{S}_1(\cH)})\subset  L^1(QP).$$
It is proved above that by the Cotlar-Stein lemma, the \textit{bornological} isomorphism $$L^1(QP)=L^1(\cR)$$ holds for $\cR = \{\rho_h:h\in\cH\}$ in the sense the the bounded sets are the same but not the topologies. Here $f\in L^1(QP)$ implies $\xchi_Kf$ is $(QP)$-Lebesgue integrable
for $K$ a compact product set. A similar argument holds for the locally convex topology $\tau_1(QP)$. Here we have a type of uniform boundedness principle
for the bimeasure $QP$.

It is remarkable that the Lebesgue norm topology of the Banach function space $L^1(QP)$ is bornologically equivalent to the much weaker locally convex $\cR$-topology and the locally convex topology $\tau_1(QP)$ for the spectral bimeasure $QP$. The Banach function space $L^1(QP)$ is an \textit{optimal domain}
for the spectral bimeasure $QP$ defined on simple functions. %see \cite{}.
\end{rem}

\section{Peller's First and Second Theorems}

Let $\cH$ be a separable Hilbert space.
Let $P:\cE\to\cL _s(\cH)$ and $Q:\cF\to\cL _s(\cH)$ be  spectral measures on 
the measurable spaces $(\L,\cE)$ and $(\G,\cF)$. Let $\f{S}=\cC_1(\cH)$ the
trace class operators on $\cH$. V. Peller %\cite{} 
obtained a Grothendieck-style characterisation
of the function space $L^1((P\ot Q)_\f{S}) = L^1((P\ot Q)_{\cL(\cH)})$. One direction is easy to see. If $\varphi$ has the decomposition
$$\varphi = \int_{T} \a(\cdot,t)\b(\cdot,t)\,d\nu(t)$$
with
$$\int_{T} \|\a(\cdot,t)\|_{L^\infty(Q)}\|\b(\cdot,t)\|_{L^\infty(P)}\, d\nu(t) < \infty,$$
then
$$\big((P\ot Q)_\f{S}(\varphi)\big)u = \int_{T}  P(\a(\cdot,t))uQ\b((\cdot,t))\, d\nu(t)$$
is trace class for $u\in\f{S}$.

In terms of Schur mulipliers of matrices, equation (\ref{eqn:Peller_norm}) shows that
the $L^1_G(m)$-norm of the function (\ref{eqn:Peller_rep}) is given by 
the expression (\ref{eqn:Peller_norm}). The appropriate extension of Grothendieck's
inequality to measure spaces gives Peller's representation
$\varphi = \int_{T} \a(\cdot,t)\b(\cdot,t)\,d\nu(t)$ for
$$\varphi\in L^1((P\ot Q)_\f{S}),$$
an earler conjecture of M. Birman and M. Solomyak \cite{BirmSol}, illustrating how Grothendieck's
deep studies \cite{} eventually solved a problem in scattering theory. The characterisation of
$L^1((P\ot Q)_\f{S})$ may be acheived via martingale convergence \cite{Jef} and a direct application of Grothendieck's fundamental metric theory \cite{Pis, Jef}

This, in turn, leads to the strict inclusion of the space $L^1((P\ot Q)_\f{S})$ in 
$L^1((P Q))$. A better handle on membership in $L^1((P\ot Q)_\f{S})$ is achieved with Peller's study of Besov space \cite{Peller, PellerI}. Membership of $L^1((PQ))$ and $L^1((QP))$ on abelian groups can also be achieved with H\"ormander's general study of symbol classes.

\section{Fundamental Results for Double Operator Integrals and Traces}

Here we sketch a few results that can be found in greater depth in \cite{SkTo}.

\begin{thm} Let $\cH$ be a separable Hilbert space.
Let $P:\cB(\R)\to\cL _s(\cH)$ and $Q:\cB(\R)\to\cL _s(\cH)$ be  spectral measures on $\R$. Let $\f{S}=\cC_p(\cH)$ for some $1\le p <\infty$
or $\f{S}=\cL (\cH)$. Suppose that $f\in L^1(\R)$ and $\varphi(\l,\mu) = \hat f(\l-\mu)$ for all  $\l,\mu\in\R$.
Then  $\int_{\R\times \R}\varphi\,d(P\ot Q)_\f{S}\in \cL (\f{S})$ and
$$\left\|\int_{\R\times \R}\varphi\,d(P\ot Q)_\f{S}\right\|_{\cL (\f{S})} \le \|f\|_1 .$$
\end{thm}

\begin{proof} For each $T\in \cC_1(\cH)$, the set function $E\times F\longmapsto P(E)T Q(F)$, $E,F\in \cB(\R)$,
is the restriction to all measurable rectangles of an $\cL (\cH)$-valued measure $\s$-additive for the strong operator
topology and the integral
\begin{align}
\int_{\R\times \R}\varphi\,d(PTQ) &= \int_{\R\times \R}\left(\int_\R e^{-it(\l-\mu)t}f(t)\,dt\right)\,d(PTQ)(\l,\mu)\cr
&= \int_\R\left(\int_{\R\times \R} e^{-it(\l-\mu)t}\,d(PTQ)(\l,\mu)\right)f(t)\,dt\cr
&=\int_\R e^{-itA}Te^{itB}f(t)\,dt\label{eqn:dopform}
\end{align}
converges as a Bochner integral in the strong operator topology to an element of
the operator ideal $\cC_1(\cH)$ of trace class operators. The interchange
of integrals is verified scalarly. 

It follows that
$\varphi$ is a $(P\ot Q)_{\cC_1(\cH)}$-integrable function and 
$$\left\|\int_{\R\times \R}\varphi\,d(P\ot Q)_{\cC_1(\cH)}\right\|_{\cL (\cC_1(\cH))} \le \|f\|_1 .$$
The corresponding bound for $\f{S}=\cC_p(\cH)$ for $1\le p <\infty$
and $\f{S}=\cL (\cH)$ follows by duality and interpolation, or directly from formula (\ref{eqn:dopform}).
\end{proof}

\begin{prp} Let $\cH$ be a separable Hilbert space and let $A,B$ be selfadjoint operators with 
spectral measures $P_A:\cB(\s(A))\to\cL _s(\cH)$ and $P_B:\cB(\s(B))\to\cL _s(\cH)$, respectively. Let $\f{S}=\cC_p(\cH)$ for some $1\le p <\infty$
or $\f{S}=\cL (\cH)$.  If the spectra of $A$ and $B$ are separated by a distance $d(\s(A),\s(B)) = \d > 0$,
then  $\int_{\s(A)\times \s(B)} ( \l-\mu)^{-1}(P_A\ot P_B)_\f{S}(d\l,d\mu)  \in \cL (\f{S})$ and
$$\left\|\int_{\s(A)\times \s(B)}\frac{(P_A\ot P_B)_\f{S}(d\l,d\mu)}{\l-\mu}\right\|_{\cL (\f{S})} \le \frac{\pi}{2\d} .$$
In particular, $AX-XB = Y$ has a unique strong solution for $Y\in\f{S}$ given by the double operator integral
$$X= \int_{\s(A)\times \s(B)}\frac{ dP_A(\l)YdP_B(\mu)  }{\l-\mu}:=\left(\int_{\s(A)\times \s(B)}\frac{(P_A\ot P_B)_\f{S}(d\l,d\mu)}{\l-\mu}\right)Y,$$
so that $\|X\|_\f{S} \le \frac{\pi}{2\d}\|Y\|_\f{S}$ \cite{BDMc}.
\end{prp}

Although the Heaviside function $\xchi_{(0,\infty)}$ is not the Fourier transform
of an $L^1$-function, the following result of I. Gohberg and M. Krein \cite[Section III.6]{GK1}
holds, in case $P=Q$.

\begin{thm} Let $\cH$ be a separable Hilbert space. 
Let $P:\cB(\R)\to\cL _s(\cH)$ and $Q:\cB(\R)\to\cL _s(\cH)$ be  spectral measures on $\R$. 
Then 
$$\int_{\R\times \R}\xchi_{\{\l>\mu\}}\,d(P\ot Q)_{\cC_p(\cH)} \in \cL (\cC_p(\cH))$$  
for every $1 < p <\infty$.
\end{thm}
The following recent result of F. Sukochev and D. Potapov (see \cite{TodTur}) settled a long outstanding
conjecture of M. Krein for the index $p$ in the range $1 < p <\infty$. 
\begin{thm}\label{thm:Krein} Let $\cH$ be a separable Hilbert space.
Let $P:\cB(\R)\to\cL _s(\cH)$ and $Q:\cB(\R)\to\cL _s(\cH)$ be  spectral measures on $\R$. 
Suppose that $f:\R\to \R$ is a continuous function for which the difference quotient
$$\varphi_f(\l,\mu) = \left\{\begin{array}{ll}
\frac{f(\l)-f(\mu)}{\l-\mu}&,\quad\l\neq\mu, \\
0&,\quad\lambda = \mu,
\end{array}\right.$$
is uniformly bounded.
Then for every $1 < p <\infty$,
$$\int_{\R\times \R}\varphi_f\,d(P\ot Q)_{\cC_p(\cH)} \in \cL (\cC_p(\cH))$$
and there exists $C_p > 0$ such that
$$\left\|\int_{\R\times \R}\varphi_f\,d(P\ot Q)_{\cC_p(\cH)}\right\|_{\cC_p(\cH)} \le C_p\|\varphi_f\|_\infty.$$
\end{thm}

Such a function $f$ is said to be \textit{uniformly Lipschitz} on $\R$
and $\|f\|_{{\rm Lip}_1} :=\|\varphi_f\|_\infty$.

\begin{cor} \label{cor:pert_Suk}
Suppose that $f:\R\to \R$ is a uniformly Lipschitz function.
Then for every $1 < p <\infty$, there exists $C_p > 0$ such that
$$\|f(A)-f(B)\|_{\cC_p(\cH)} \le C_p\|f\|_{{\rm Lip}_1}\|A-B\|_{\cC_p(\cH)}$$
for any selfadjoint operators $A$ and $B$ on a separable Hilbert space $\cH$.
\end{cor}

\begin{proof} Let $P_A$ and $P_B$ be the spectral meaures of $A$ and $B$, respectively, and suppose that
$\|A-B\|_{\cC_p(\cH)} <\infty$. Then according to \cite[Theorem 8.1]{BirmSol} (see also \cite[Corollary 7.2]{dPWS}),
the equality
$$f(A)-f(B) =  \left(\int_{\R\times \R}\varphi_f\,d(P_A\ot P_B)_{\cC_p(\cH)}\right)(A-B)$$
holds and the norm estimate follows from Theorem \ref{thm:Krein}.
\end{proof}

\section{An elementary proof of Krein's spectral trace formula}

We refer to the survey paper of M Birman and M. Solomyak \cite{BirmSol} . According to their
history of the subject, they hoped to prove Krein's formula just with general
double operator integral arguments, which go back to old work of
Yu. Daletskii and M. Krein in the 1950's, see \cite{BirmSol}. They were aware of the arguments of Feynman's Operational Calculus which has received a new boost in the monograph \cite{JohnLap}
of M. Lapidus, L. Nielsen and the late G.W. Johnson.

The proof below is adapted from \cite{Boy} which has an attractive and elementary Fourier Theory argument in the upper half-plane.  Other proofs  use elementary Complex Analysis but all finally appeal to a Bootstrap estimate on eigenvalues of a positive trace class operator. The original argument was simply by analogy from the case of matrices and determinants, which is
often a good testing ground for infinite dimensional operator perturbation theory, see \cite{Kato}.

\begin{thm} Let $\cH$ be a separable Hilbert space
and let $A$ and $B$ be selfadjoint operators with the same domain such that $A-B\in \cC_1(\cH)$. Then
there exists a spectral shift function $\xi\in L^1(\R)$ such that
\begin{equation}\label{eqn:Krein_form}
\hbox{\rm tr}(f(A) - f(B)) = \int_\R f'(\l)\xi(\l)\,d\l
\end{equation}
for every function $f:\R\to \C$ for which there exists a finite positive Borel measure $\mu$ on $\R$ such that 
$$f(x) = i\int_\R \frac{e^{-isx}-1}{s}\,d\mu(s),\quad x\in\R.$$
Furthermore, $\xi$ possesses the following properties.
\begin{enumerate}
\item[\rm a)] $\hbox{\rm tr}(A - B) = \int_\R \xi(\l)\,d\l$.
\item[\rm b)] $\|\xi\|_1\le \|A-B\|_{\cC_1(\cH)}$.
\item[\rm c)] If $B\le A$, then $\xi\ge 0$ a.e.
\item[\rm d)] $\xi$ is zero a.e. outside the interval $(\inf(\s(A)\cup\s(B)),\sup(\s(A)\cup\s(B)))$. 
\end{enumerate}
\end{thm}

\begin{proof} 
It suffices to assume that $A$ and $B$ are hermitian matrices.
The estimate 
\begin{equation}\label{eqn:trbd}
\|f(A)-f(B)\|_{\cC_1(\cH)}\le \mu(\R)\|A-B\|_{\cC_1(\cH)}
\end{equation}follows from the bound of Corollary \ref{cor:pert_Suk}
and the calculation
$$f(A)-f(B) = \frac{i}{2\pi}\int_\R \frac{e^{-isA}-e^{-isB}}{s}\,d\mu(s)$$
obtained from an application of Fubini's theorem with respect to $P_A\ot\mu$ and  $P_B\ot\mu$ on $\R\times [\e,\infty)$ for $\e > 0$.
Then 
$$\hbox{\rm tr}(f(A) - f(B)) = \frac{1}{2\pi}\int_\R \Phi \,d\mu.$$

An expression for the spectral shift function $\xi$ may be obtained from Fatou's Theorem \cite[Theorem 11.24]{Rudin}. Suppose that $\nu$ is a finite measure on $\R$
$$\phi_\nu(z) =\frac{1}{2\pi i}\int_\R \frac{d\nu(\l)}{\l-z},\quad z\in\C\setminus \R,$$ 
is the Cauchy transform of $\nu$. Then $\nu$ is absolutely continuous if
$$\hat\nu(\xi) = \int_\R e^{-i\xi x}(\phi_\nu(x+i0+)-\phi_\nu(x+i0-))\,dx,\quad \xi\in\R.$$
The function $x\longmapsto \phi_\nu(x+i0+)-\phi_\nu(x+i0-)$ defined for almost all $x\in\R$
is then the density of $\nu$ with respect to Lebesgue measure.
For $\nu =\Xi $, if the representation
\begin{align*}
\Phi(s) &= i\frac{\hbox{tr}(e^{-isA }-e^{-isB})}{s}\\
&=\frac{1}{2\pi i}\int_\R e^{-isx}\bigg(\lim_{\e\to0+}\int_0^1\hbox{tr} (V(A+tV-x-i\e)^{-1}-\\ 
&\qquad V(A+tV-x+i\e)^{-1})\,dt\bigg)dx 
\end{align*}
were valid, we would expect that $\xi=\check\Phi$ has the representation
\begin{align}
\xi(s) &=\frac{1}{2\pi i }\lim_{\e\to0+}\int_\R e^{isx-\e|x|} \frac{\hbox{tr}(e^{-ixA }-e^{-ixB})}{x} dx,\quad s\in\R,\nonumber\\
&=\lim_{\e\to0+}\frac1\pi\hbox{tr}\left[\arctan\left(\frac{A-s I}\e\right)-\arctan\left(\frac{B-s I}\e\right)\right],\label{eqn:Krein_rep}
\end{align}
where the arctan function may be expressed as
\begin{equation}\label{eqn:arctanrep}
\arctan t=\frac{1}{2i}\int_\R\frac{e^{ist}-1}{s}e^{-|s|}\,ds,\quad t\in\R .
\end{equation}
For the function defined by
\begin{align*}
h(x,y) &=\frac1\pi\hbox{tr}\left[\arctan\left(\frac{A-x I}y\right)-\arctan\left(\frac{B-x I}y\right)\right]
\end{align*}
we have the bounds
\begin{align*}
\pi |h(x,y)| &\le \left\|\arctan\left(\frac{A-x I}y\right)-\arctan\left(\frac{B-x I}y\right)\right\|_{\cC_1(\cH)}\\
&\le \frac1y\|A-B\|_{\cC_1(\cH)},
\end{align*}
from the bound (\ref{eqn:trbd}) and the representation (\ref{eqn:arctanrep}).
Rewriting
$$h(x,y) =  \frac1{2\pi i}\int_\R e^{-ixs-y|s|}\hbox{tr}\left[\frac{e^{isA}-e^{isB}}{s}\right]ds$$
using (\ref{eqn:arctanrep}), it follows that $h(x,y)$ is harmonic in the upper half-plane 
$$\{(x,y):x\in\R,\ y>0\}.$$

We first look at the case that $A-B = \a (\cdot,w)w$ for $\a > 0$ and $w\in \cH$, $\|w\|=1$, 
so that $A$ is a rank one perturbation of the bounded selfadjoint operator $B$.

If we set
$$X = 2\arctan \frac{A-x}{y}, Y  = 2\arctan \frac{B-x}{y},$$
then $2\pi h = \hbox{tr}(X-Y)$. The formula tr\,$\log(e^{iX}e^{-iY}) = i\hbox{tr}(X-Y)$
follows from the Baker-Campbell-Hausdorff formula for large $y > 0$, see \cite[Lemma 1.1]{Boy}.
Let $T_A = e^{-iX}$, $T_B= e^{-iY}$. Then for $z = x+iy$, 
elementary spectral theory gives
$$
\begin{array}{ll}
T_A&=(A-\overline zI)(A-zI)^{-1} = I + 2iy(A-z I)^{-1},\\
T_B&= (B-\overline zI)(B-zI)^{-1} = I + 2iy(B-z I)^{-1}.
\end{array}
$$
Our aim is to compute tr\,$\log(U)$ for the unitary operator $U = T_A^*T_B$. Because
\begin{align*}
U-I &= T_A^*T_B -T_B^*T_B\\
&= (T_A^* -T_B^*)T_B\\
&=-i2y[(A-\overline z I)^{-1}-(B-\overline z I)^{-1}] T_B,
\end{align*}
we obtain
$$U = I + i2y(A-\overline z I)^{-1}(A-B)(B-z I)^{-1}.$$
Substituting $A-B = \a (\cdot,w)w$ gives
$$U = I + i2y\a(\cdot,(B-\overline z I)^{-1}w)   (A-\overline z I)^{-1}w.$$

The vector $(A-\overline z I)^{-1}w$ is an eigenvector for the unitary operator $U$
with eigenvalue
$$1+ i2y\a((A-\overline z I)^{-1}w,(B-\overline z I)^{-1}w)$$
which can be expressed as $e^{i2\pi\theta(x,y) }$ for a continuous function $\theta$ in the upper half plane
such that $0 < \theta < 1$. Consequently, for large $y > 0$, 
$$i2\pi\theta = \hbox{tr}\log(U) = i\hbox{tr}(X-Y) =i2\pi h.$$
Then $\theta$ is harmonic for large $y > 0$ so it is harmonic
on the upper half plane and it is equal to $h$ there, so $0 < h < 1$.
For matrices, the formula for $h$ is quite explicit and we can proceed by
direct calculation.

By Fatou's Theorem, the boundary values $\xi(x)=\lim_{y\to0+}h(x,y)$ are defined for
almost all $x\in\R$ and satisfy  
$$\lim_{y\to\infty}\pi y h(x,y) = \int_\R\xi(t)\,dt =\|\xi\|_1 \le \|A-B\|_{\cC_1(\cH)}$$
for every $x\in\R$, so in the case that $A-B$ has rank one, formula (\ref{eqn:Krein_rep}) is valid.

For an arbitrary selfadjoint perturbation $V =\sum_{j=1}^\infty\a_j(\cdot,w_j)w_j$ with 
$$\sum_{j=1}^\infty|\a_j| = \|A-B\|_{\cC_1(\cH)} <\infty,$$
the function $\xi_n\in L^1(\R)$ may be defined in a similar fashion for $A_n = B + \sum_{j=1}^n\a_j(\cdot,w_j)w_j$, $n=1,2,\dots$, so that $\xi_n\to \xi$ in $L^1(\R)$
as $n\to\infty$ from which it verified that $\xi = \check\Phi$.
\end{proof}

The representation $\xi = \check\Phi$ obtained above may be viewed as the Fourier transform approach. In the case of a rank one perturbation
$V= \a (\cdot,w)w$, the Cauchy transform approach is developed by B. Simon \cite{Si95}
with the formula
$$\hbox{\rm tr}((A-zI)^{-1} - (B-zI)^{-1}) = -\int_\R  \frac{\xi(\l)}{(\l-z)^2}\,d\l$$
for $z\in\C\setminus[a,\infty)$ for some $a\in\R$, established in \cite[Theorem 1.9]{Si95} by computing
a contour integral. Here the boundary value $\xi(x)=\lim_{y\to0+}h(x,y)$ is expressed as
$$\xi(x) = \frac1\pi \hbox{Arg}(1+\a F(\l+i0+))$$
for almost all $x\in\R$ with respect to the Cauchy transform
$$F(z) = \int_\R \frac{d(P_Bw,w)(\l)}{\l-z},\quad z\in \C\setminus(-\infty,a).$$
The Cauchy transform approach is generalised to type II von Neumann algebras in \cite{ADS}.

Many different proofs of Krein's formula (\ref{eqn:Krein_form}) are available for a wide class of functions $f$, especially
in a form that translates into the setting of noncommutative integration \cite{ADS,PellerI}. As remarked in \cite[p. 163]{BirmSol},
an ingredient additional to double operator integrals (such as complex function theory) 
is needed to show that the measure $\Xi $ is absolutely continuous with respect to Lebesgue measure on $\R$.
Krein's original argument uses perturbation determinants from which follows the representation
Det$(S(\l)) = e^{-2\pi i\xi(\l)}$ for the scattering matrix $S(\l)$ for $A$ and $B$ \cite[Chapter 8]{Y} ,which is still a basic tool of scattering theory.

%\input appendix.tex

% for BibTeX users
%\bibliographystyle{ws-book-har}    % Bibliography: Author-Date system
%\bibliography{ws-book-sample}      % pls. call your database here

% for non-BibTeX users
 %\input bibliography.tex

%\title{Index}                      % to set `Index` as even page running title
\printindex
\end{document}